\documentclass[12pt,sloppy]{article}
\usepackage{graphicx}
\usepackage{subfigure}
\usepackage{latexsym}
\usepackage{amsfonts}
\usepackage{amsthm}
\usepackage{amsmath}
\usepackage{caption}
\usepackage{chngpage}
\author{{\bf Maya Mohsin Ahmed }}
\title{{\LARGE {\bf Unraveling the secret of Benjamin Franklin: \\  Constructing Franklin squares of higher order }}}
\date{}

\newtheorem{prop}{Proposition}[section]

\newtheorem{defn}{Definition}[section]

\begin{document}     % this begins the actual document body
\maketitle           % this actually creates the title block

\begin{abstract}
Benjamin  Franklin  constructed  three  squares  which  have
amazing properties, and his  method of construction has been
a mystery  to date. In  this article, we divulge  his secret
and show how to construct such squares for any order.
\end{abstract}

\section{Introduction}

%%%% Figure of Franklin squares %%%%%%%%%%%%%%%%%%%%%%%
\begin{figure}[h]
 \begin{center}
     \includegraphics[scale=0.5]{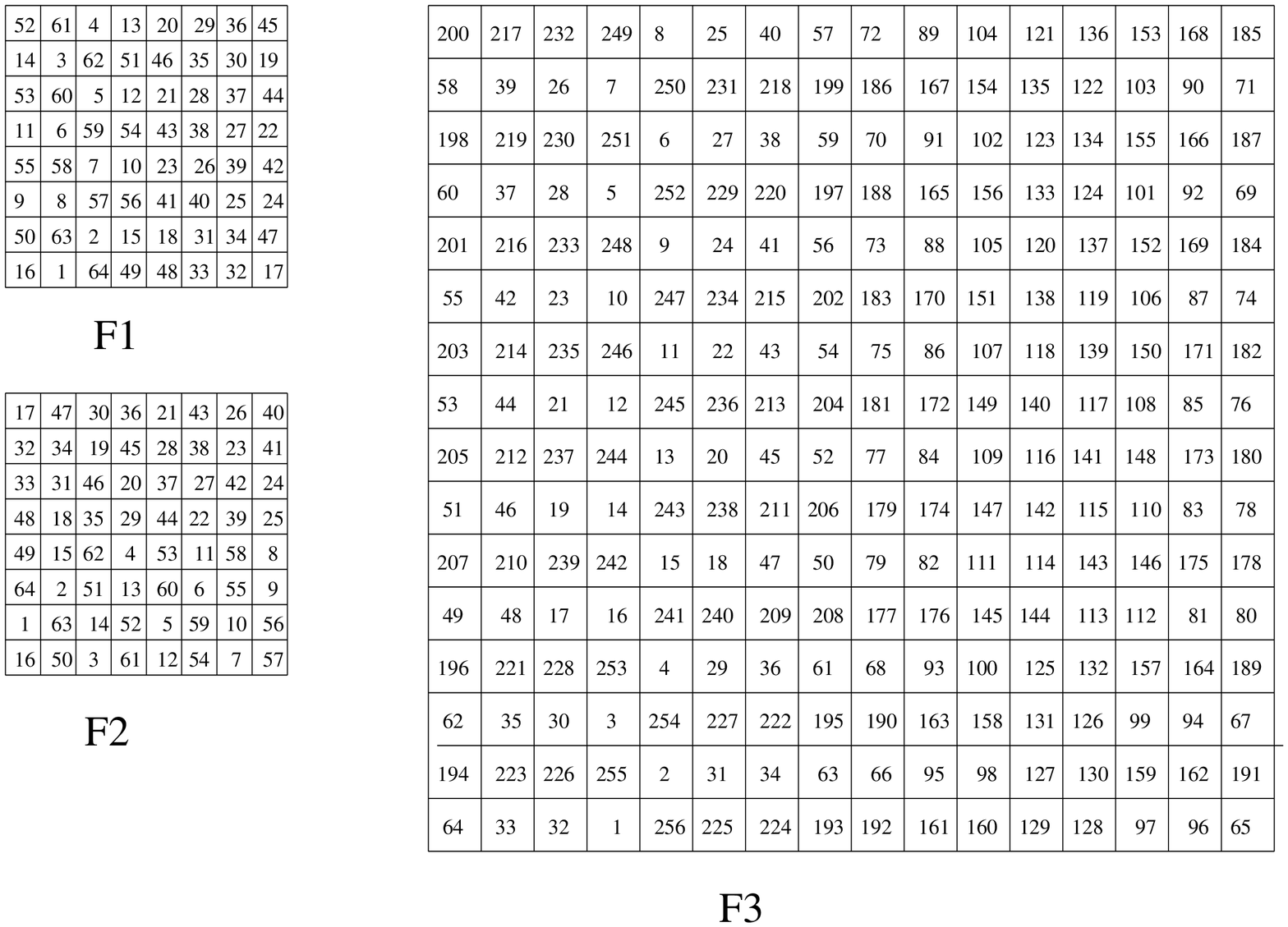}
\caption{Squares constructed by Benjamin Franklin} \label{franklinsquares}
 \end{center}
 \end{figure}
\begin{figure}[h]
 \begin{center}
     \includegraphics[scale=0.3]{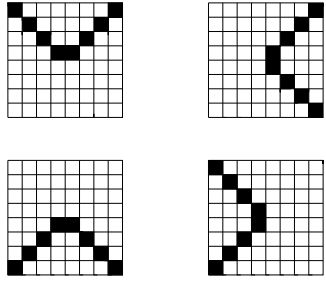}
\caption{The four main bend diagonals \cite{pasles}} \label{bent_diags}
 \end{center}
 \end{figure}
%%%%%%% parallel bend diagonals %%%%%%%%%%%%%%%%%
\begin{figure}[h]
 \begin{center}
     \includegraphics[scale=0.4]{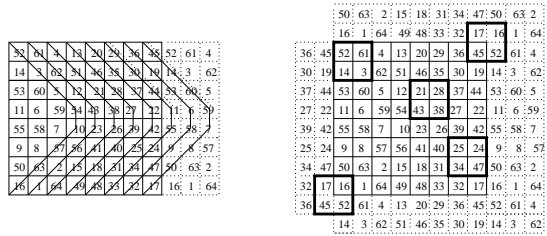}
\caption{Continuous properties of Franklin squares} \label{parallelbends}
 \end{center}
 \end{figure}
%%%%%%%%%%%%%%%%%%%%%%%%%%%%%%%%%%%%%%%%%%%%%%%%

The well-known squares in Figure \ref{franklinsquares} were
constructed by Benjamin Franklin (see \cite{andrews} and
\cite{pasles}). In a letter to Peter Collinson (see \cite{andrews}) he
describes the properties of the $8 \times 8$ square, F1 as:
\begin{enumerate}
\item The entries of every row and column add to a common sum
called the {\em magic sum}. 
\item In every half row and half column the entries add to half the magic sum.
\item The entries of the main bend diagonals (see Figure
\ref{bent_diags}) and all the bend diagonals parallel to it (see
Figure \ref{parallelbends}) add to the magic sum.
\item The four corner numbers with the four middle numbers add to the magic
sum. 
\end{enumerate}

Franklin mentions that the  square F1 has five other curious
properties without  listing them. He also says,  in the same
letter,  that  the $16  \times  16$  square,  F3, in  Figure
\ref{franklinsquares}  has  all  the  properties of  the  $8
\times  8$  square, but  in  addition,  every  $4 \times  4$
sub-square adds to the common  magic sum. 

We now  bring to attention some important  facts about these
squares. The entries of the squares are from the set $\{1,2,
\dots, n^2\}$, where $n=8$  or $n=16$. Every integer in this
set  occurs in  the square  exactly once.  The $8  \times 8$
squares have magic sum 260 and the $16 \times 16$ square has
magic sum 2056.  Unlike, F1 or F2, observe  that for F3, the
four corner numbers with the four middle numbers add to half
the magic sum. In addition,  observe that every $2 \times 2$
sub-square in F1  and F2 adds to half the  magic sum, and in
F3 adds to one-quarter the magic sum. The property of the $2
\times  2$  sub-squares  adding  to  a common  sum  and  the
property of bend diagonals adding  to the magic sum are {\em
  continuous  properties}. By  continuous  property we  mean
that if  we imagine  the square is  the surface of  a torus;
i.e. opposite  sides of the square are  glued together, then
the bend  diagonals or the  $2 \times 2$ sub-squares  can be
translated and still the corresponding sums hold (see Figure
\ref{parallelbends}). It  is worth noticing  that the fourth
property listed by Benjamin Franklin is redundant because of
the  continuous  property of  the  $2  \times 2$  sub-squares
adding  to  a  common   sum.  Moreover,  the  $2  \times  2$
sub-square  property   implies  that  every   $4  \times  4$
sub-square adds to the magic sum in F3. For a detailed study
of these three squares constructed by Benjamin Franklin, see
\cite{ma,andrews,pasles}.

From now on, row sum, column sum, or bend diagonal sum, etc.
mean that  we are adding  the entries of those  elements. We
use the  description provided  by Benjamin Franklin  and our
observations to define {\em Franklin squares}.

\begin{defn}[Franklin Square]Consider
 an integer, $n=2^r$ such that $r \geq 3$. Let the magic sum
 be  denoted by  $M$ and  $N=n^2+1$.  We define  an {\em  $n
   \times n$  Franklin square} to  be a $n \times  n$ matrix
 with the following properties:
\begin{enumerate}
\item  Every  integer from  the  set  $\{1,2, \dots,  n^2\}$
  occurs  exactly once  in the  square. Consequently,  \[M =
  \frac{n}{2} N.\]
\item All  the the half  rows, half columns add  to one-half
  the magic sum. Consequently,  all the rows and columns add
  to the magic sum.
\item All the bend diagonals add to the magic sum, continuously.

\item All  the $2  \times 2$ sub-squares  add to  $4M/n =2N$,
  continuously.   Consequently,  all   the   $4  \times   4$
  sub-squares add to $8N$, and the four corner numbers
  with the four middle numbers add to $4N$.
\end{enumerate}
\end{defn}

In \cite{ma}, we described how to construct Franklin squares
using Algebraic  Geometry and Combinatorics.  Those methods,
being  computationally  challenging,  are not  suitable  for
higher   orders.  In  this   article,  we   follow  Benjamin
Franklin's   footsteps  closely,   and   provide  elementary
techniques  to  construct a  Franklin  square  of any  given
order. The  strategies of construction of  Franklin squares F1
and F3 are  the same and our methods  are based on observing
these two squares. The  construction of F2 is different from
F1 and we  do not touch upon the construction  of F2 in this
article. A $32 \times 32$  Franklin square is given in Table
\ref{32franklineg}.  A  Franklin square  of  this order  has
never been  constructed before. Moreover, the  Maple code we
provide  in Section  \ref{codesection} can  easily construct
Franklin squares of higher  order. The details of the method
are given in in Section \ref{constructionsection}.

\newpage
{\tiny
\captionof{table}{A $32 \times 32$ Franklin square.} \label{32franklineg}
\begin{adjustwidth}{-1in}{.5in}
\[
\begin{array}{|c|c|c|c|c|c|c|c|c|c|c|c|c|c|c|c|c|c|c|c|c|c|c|c|c|c|c|c|c|c|c|c|} \hline
784 & 817 & 848 & 881 & 912 & 945 & 976 & 1009 & 16 &   49 &  80 &  113 & 144 & 177 & 208 & 241 
& 272 & 305 & 336 & 369 & 400 & 433 & 464 & 497 & 528 & 561 & 592 & 625 & 656 & 689 & 720 & 753  \\ \hline  
242 & 207 & 178 & 143 & 114 & 79 &  50 &  15 &   1010 & 975 & 946 & 911 & 882 & 847 & 818 & 783
& 754 & 719 & 690 & 655 & 626 & 591 & 562 & 527 & 498 & 463 & 434 & 399 & 370 & 335 & 306 & 271 \\ \hline
782 & 819 & 846 & 883 & 910 & 947 & 974 & 1011 & 14 &   51 &  78 &  115 & 142 & 179 & 206 & 243 
& 270 & 307 & 334 & 371 & 398 & 435 & 462 & 499 & 526 & 563 & 590 & 627 & 654 & 691 & 718 & 755 \\ \hline
244 & 205 & 180 & 141 & 116 & 77 &  52 &  13  &  1012 & 973 & 948 & 909 & 884 & 845 & 820 & 781 
& 756 & 717 & 692 & 653 & 628 & 589 & 564 & 525 & 500 & 461 & 436 & 397 & 372 & 333 & 308 & 269   \\ \hline
780 & 821 & 844 & 885 & 908 & 949 & 972 & 1013 & 12 &   53  & 76 &  117 & 140 & 181 & 204 & 245 
& 268 & 309 & 332 & 373 & 396 & 437 & 460 & 501 & 524 & 565 & 588 & 629 & 652 & 693 & 716 & 757 \\ \hline
246 & 203 & 182 & 139 & 118 & 75 &  54  & 11 &   1014 & 971 & 950 & 907 & 886 & 843 & 822 & 779
&758 & 715 & 694 & 651 & 630 & 587 & 566 & 523 & 502 & 459 & 438 & 395 & 374 & 331 & 310 & 267  \\ \hline
778 & 823 & 842 & 887 & 906 & 951 & 970 & 1015 & 10 &   55  & 74  & 119 & 138 & 183 & 202 & 247
& 266 & 311 & 330 & 375 & 394 & 439 & 458 & 503 & 522 & 567 & 586 & 631 & 650 & 695 & 714 & 759  \\ \hline
248 & 201 & 184 & 137 & 120 & 73  & 56  & 9    & 1016 & 969 & 952 & 905 & 888 & 841 & 824 & 777
&760 & 713 & 696 & 649 & 632 & 585 & 568 & 521 & 504 & 457 & 440 & 393 & 376 & 329 & 312 & 265   \\ \hline
785 & 816 & 849 & 880 & 913 & 944 & 977 & 1008 & 17   & 48 &  81 &  112 & 145 & 176 & 209 & 240
&273 & 304 & 337 & 368 & 401 & 432 & 465 & 496 & 529 & 560 & 593 & 624 & 657 & 688 & 721 & 752  \\ \hline 
239 & 210 & 175 & 146 & 111 & 82 &  47 &  18 &   1007 & 978 & 943 & 914 & 879 & 850 & 815 & 786 
&751 & 722 & 687 & 658 & 623 & 594 & 559 & 530 & 495 & 466 & 431 & 402 & 367 & 338 & 303 & 274 \\ \hline
787 & 814 & 851 & 878 & 915 & 942 & 979 & 1006 & 19   & 46 &  83  & 110 & 147 & 174 & 211 & 238
& 275 & 302 & 339 & 366 & 403 & 430 & 467 & 494 & 531 & 558 & 595 & 622 & 659 & 686 & 723 & 750  \\ \hline
237 & 212 & 173 & 148 & 109 & 84  & 45 &  20 &   1005 & 980 & 941 & 916 & 877 & 852 & 813 & 788
&749 & 724 & 685 & 660 & 621 & 596 & 557 & 532 & 493 & 468 & 429 & 404 & 365 & 340 & 301 & 276   \\ \hline
789 & 812 & 853 & 876 & 917 & 940 & 981 & 1004 & 21   & 44 &  85  & 108 & 149 & 172 & 213 & 236
&277 & 300 & 341 & 364 & 405 & 428 & 469 & 492 & 533 & 556 & 597 & 620 & 661 & 684 & 725 & 748  \\ \hline
235 & 214 & 171 & 150 & 107 & 86 &  43  & 22   & 1003 & 982 & 939 & 918 & 875 & 854 & 811 & 790
&747 &726 & 683 & 662 & 619 & 598 & 555 & 534 & 491 & 470 & 427 & 406 & 363 & 342 & 299 & 278   \\ \hline
791 & 810 & 855 & 874 & 919 & 938 & 983 & 1002 & 23  &  42  & 87  & 106 & 151 & 170 & 215 & 234
&279 & 298 & 343 & 362 & 407 & 426 & 471 & 490 & 535 & 554 & 599 & 618 & 663 & 682 & 727 & 746  \\ \hline
233 & 216 & 169 & 152 & 105 & 88 &  41  & 24 &   1001 & 984 & 937 & 920 & 873 & 856 & 809 & 792
&745 & 728 & 681 & 664 & 617 & 600 & 553 & 536 & 489 & 472 & 425 & 408 & 361 & 344 & 297 & 280  \\ \hline
793 & 808 & 857 & 872 & 921 & 936 & 985 & 1000 & 25 &   40  & 89  & 104 & 153 & 168 & 217 & 232
& 281 & 296 & 345 & 360 & 409 & 424 & 473 & 488 & 537 & 552 & 601 & 616 & 665 & 680 & 729 & 744  \\ \hline
231 & 218 & 167 & 154 & 103 & 90 &  39  & 26   & 999 &  986 & 935 & 922 & 871 & 858 & 807 & 794
&743 & 730 & 679 & 666 & 615 & 602 & 551 & 538 & 487 & 474 & 423 & 410 & 359 & 346 & 295 & 282  \\ \hline
795 & 806 & 859 & 870 & 923 & 934 & 987 & 998 &  27  &  38 &  91  & 102 & 155 & 166 & 219 & 230
& 283 & 294 & 347 & 358 & 411 & 422 & 475 & 486 & 539 & 550 & 603 & 614 & 667 & 678 & 731 & 742 \\ \hline
229 & 220 & 165 & 156 & 101 & 92 &  37  & 28  &  997 &  988 & 933 & 924 & 869 & 860 & 805 & 796
& 741 & 732 & 677 & 668 & 613 & 604 & 549 & 540 & 485 & 476 & 421 & 412 & 357 & 348 & 293 & 284 \\ \hline
797 & 804 & 861 & 868 & 925 & 932 & 989 & 996 &  29  &  36  & 93  & 100 & 157 & 164 & 221 & 228
&285 & 292 & 349 & 356 & 413 & 420 & 477 & 484 & 541 & 548 & 605 & 612 & 669 & 676 & 733 & 740 \\ \hline
227 & 222 & 163 & 158 & 99  & 94  & 35  & 30  &  995 &  990 & 931 & 926 & 867 & 862 & 803 & 798
& 739 & 734 & 675 & 670 & 611 & 606 & 547 & 542 & 483 & 478 & 419 & 414 & 355 & 350 & 291 & 286  \\ \hline
799 & 802 & 863 & 866 & 927 & 930 & 991 & 994 &  31  &  34  & 95  & 98  & 159 & 162 & 223 & 226 
&287 & 290 & 351 & 354 & 415 & 418 & 479 & 482 & 543 & 546 & 607 & 610 & 671 & 674 & 735 & 738  \\ \hline
225 & 224 & 161 & 160 & 97  & 96 &  33  & 32  &  993 &  992 & 929 & 928 & 865 & 864 & 801 & 800 
& 737 & 736 & 673 & 672 & 609 & 608 & 545 & 544 & 481 & 480 & 417 & 416 & 353 & 352 & 289 & 288 \\ \hline
776 & 825 & 840 & 889 & 904 & 953 & 968 & 1017 & 8 &    57  & 72  & 121 & 136 & 185 & 200 & 249 
& 264 & 313 & 328 & 377 & 392 & 441 & 456 & 505 & 520 &  569 & 584 & 633 & 648 & 697 & 712 & 761  \\ \hline
250 & 199 & 186 & 135 & 122 & 71  & 58  & 7    & 1018 & 967 & 954 & 903 & 890 & 839 & 826 & 775 
& 762 & 711 & 698 & 647 & 634 & 583 & 570 & 519 & 506 & 455 & 442 & 391 & 378 & 327 & 314 & 263 \\ \hline
774 & 827 & 838 & 891 & 902 & 955 & 966 & 1019 & 6 &    59  & 70  & 123 & 134 & 187 & 198 & 251
& 262 & 315 & 326 & 379 & 390 & 443 & 454 & 507 & 518 & 571 & 582 & 635 & 646 & 699 & 710 & 763 \\ \hline
252 & 197 & 188 & 133 & 124 & 69  & 60  & 5    & 1020 & 965 & 956 & 901 & 892 & 837 & 828 & 773
& 764 & 709 & 700 & 645 & 636 & 581 & 572 & 517 & 508 & 453 & 444 & 389 & 380 & 325 & 316 & 261 \\ \hline
772 & 829 & 836 & 893 & 900 & 957 & 964 & 1021 & 4 &    61  & 68  & 125 & 132 & 189 & 196 & 253
&260 & 317 & 324 & 381 & 388 & 445 & 452 & 509 & 516 & 573 & 580 & 637 & 644 & 701 & 708 & 765 \\ \hline
254 & 195 & 190 & 131 & 126 & 67 &  62  & 3    & 1022 & 963 & 958 & 899 & 894 & 835 & 830 & 771
& 766 & 707 & 702 & 643 & 638 & 579 & 574 & 515 & 510 & 451 & 446 & 387 & 382 & 323 & 318 & 259 \\ \hline
770 & 831 & 834 & 895 & 898 & 959 & 962 & 1023 & 2  &   63  & 66  & 127 & 130 & 191 & 194 & 255
&258 & 319 & 322 & 383 & 386 & 447 & 450 & 511 & 514 & 575 & 578 & 639 & 642 & 703 & 706 & 767 \\ \hline
256 & 193 & 192 & 129 & 128 & 65  & 64  & 1    & 1024 & 961 & 960 & 897 & 896 & 833 & 832 & 769
&768 & 705 & 704 & 641 & 640 & 577 & 576 & 513 & 512 & 449 & 448 & 385 & 384 & 321 & 320 & 257 \\ \hline
\end{array} \]
\end{adjustwidth}
}

\section{Method to construct Franklin squares.} \label{constructionsection}

In this section, we describe a process to construct Franklin
squares. The squares F1 and  F3 can also be constructed with
this method.

 A  Franklin square is  constructed in  two parts:  the left
 side consisting  of the first  $n/2$ columns and  the right
 side consisting of the last $n/2$ columns. The construction
 of both  the parts are  largely independent of  each other.
 Each side  is further  divided in to  three parts:  the top
 part consisting  of the first  $n/4$ rows, the  middle part
 consisting of  the middle $n/2$  rows, and the  bottom part
 consisting of the last $n/4$ rows. Throughout this article,
 reference to  these parts mean  these blocks of  rows. Each
 side is constructed by  adding pairs of columns equidistant
 from its center.  For a given part, and  a pair of columns,
 $ca$ and  $cs$, there are  two operations to  make entries,
 which  we call  Up and  Down. Let  $p$ denote  a  part, $e$
 denote the  number of  rows in $p$,  and let $rf$  and $rl$
 denote the first and last row of $p$, respectively. Let $s$
 denote a starting number. Throughout this article, from now
 on, $N=n^2+1$. The inputs to  the functions Up and Down are
 $N$, $p$, $ca$, $cs$, $s$, and $e$.

The function Up assigns values to the columns
$ca$ and $cs$ in the following manner.

For $i=0$ to $e-1$: $t=s+2i$ and
\[ \begin{array}{llllllllll}
  A_{rl-2i,cs} = t,  A_{rl-2i,ca} = N-t, A_{rl-1-2i,cs} = N-(t+1), A_{rl-1-2i,ca} = t+1. 
\end{array}
\]

The function Down assigns values to the square in the manner described below.

For $i=0$ to $e-1$: $t=s+2i$ and 
\[ \begin{array}{llllllllll}
A_{rf+2i,cs} = N-t,
A_{rf+2i,ca}= t,
A_{rf+1+2i,cs} = t+1,
A_{rf+1+2i,ca}= N-(t+1). 
\end{array}
\]

Observe that the Up function goes up the rows of the part in
a zig-zag fashion assigning values $t$ and $N-t$. Similarly,
the Down  function goes down  the rows of the  part. Observe
that these assignments guarantee that $A_{r,ca}+A_{r,cs}=N$,
for any row $r$ and pair of columns $ca,cs$.

The indices of column pairings and order of construction for
the left  side is $n/4-k$  and $n/4+1+k$; and for  the right
side is $n/2+1+k$ and  $n-k$ where $k=0, \dots, n/4-1$. This
order  of  construction  implies   that  the  left  side  is
constructed from the center going outwards whereas the right
side is  constructed from the outside  columns going inwards
to the center. The  numbers $1,2, \dots, n^2/4$ always appear
on  the left  side of  the Franklin  square and  the numbers
$n^2/4+1,  \dots, n^2/2$  appear  on the  right  side of  the
square.

We first construct the left  side of the Franklin square. To
begin, we do an Up  operation on the bottom part with $s=1$,
$cs =  n/4$, $ca=n/4+1$,  and $e=n/8$. Then,  we do  an Up
operation  on  the top  part  with  $s=n/4+1$,  $cs =  n/4$,
$ca=n/4+1$, and  $e=n/8$. Next, we do a  Down operation on
the middle part with  $s=n/2+1$, $cs = n/4$, $ca=n/4+1$, and
$e=n/4$. At  this point,  we have constructed  two columns
which contain the  integers $1, 2, \dots, n$  and $N-1, N-2,
\dots,N-n$.  Now, we do an  Up operation on  the middle part
with $s=n+1$, $cs = n/4-1$, $ca=n/4+2$, and $e=n/4$. Next,
we do a Down operation  on the top part with $s=3n/2+1$, $cs
= n/4-1$,  $ca=n/4+2$, and $e=n/8$. Finally, we  do a Down
operation on the bottom  part with $s=7n/4+1$, $cs = n/4-1$,
$ca=n/4+2$, and $e=n/8$.

Observe  that  four  columns  were created  with  the  above
operations. This sequence  of operations is repeated $n/4-1$
times to  complete the  left side of  the square,  where the
parameters  in the Up  and Down  operations are  replaced by
$s=s+2jn$,  $ca =  ca+2j$, $cs=cs-2j$,  where  $j=1, \dots,
n/4-1$.

The same  sequence of  operations and number  of repetitions
are used to construct the right side of the Franklin square.
The  parameters   $r$  and  $e$  remain   the  same  whereas
$s=s+(n/2)(n/2)$  for all  the Up  and Down  operations. The
changes in  the column parameters  $ca$ and $cs$  are listed
below.

\begin{tabular}{cccccccccccccccccccc}
 & Up(bottom) & Up(top) & Down(middle) & Up(middle) & Down(top) & Down(bottom) \\
ca: & n/2+1 &  n/2+1 &  n/2+1 & n/2+2& n/2+2& n/2+2 \\
cs: &  n & n & n & n-1 &n-1 &n-1  
\end{tabular}

See Figure \ref{leftfranklinfig} and Table \ref{tableleftf3}
for a  pictorial construction  of the left  side of  the $16
\times 16$ Franklin square F3. The pictorial construction of
right    side    of    F3    is    described    in    Figure
\ref{rightfranklinfig}  and  Table \ref{rightfranklintable}.
Maple procedures to construct  Franklin squares of any order
$n$ are provided in Section \ref{codesection}.

\begin{figure}[h]
 \begin{center}
\caption{Constructing the left side of the Franklin square F3.}
     \includegraphics[scale=0.3]{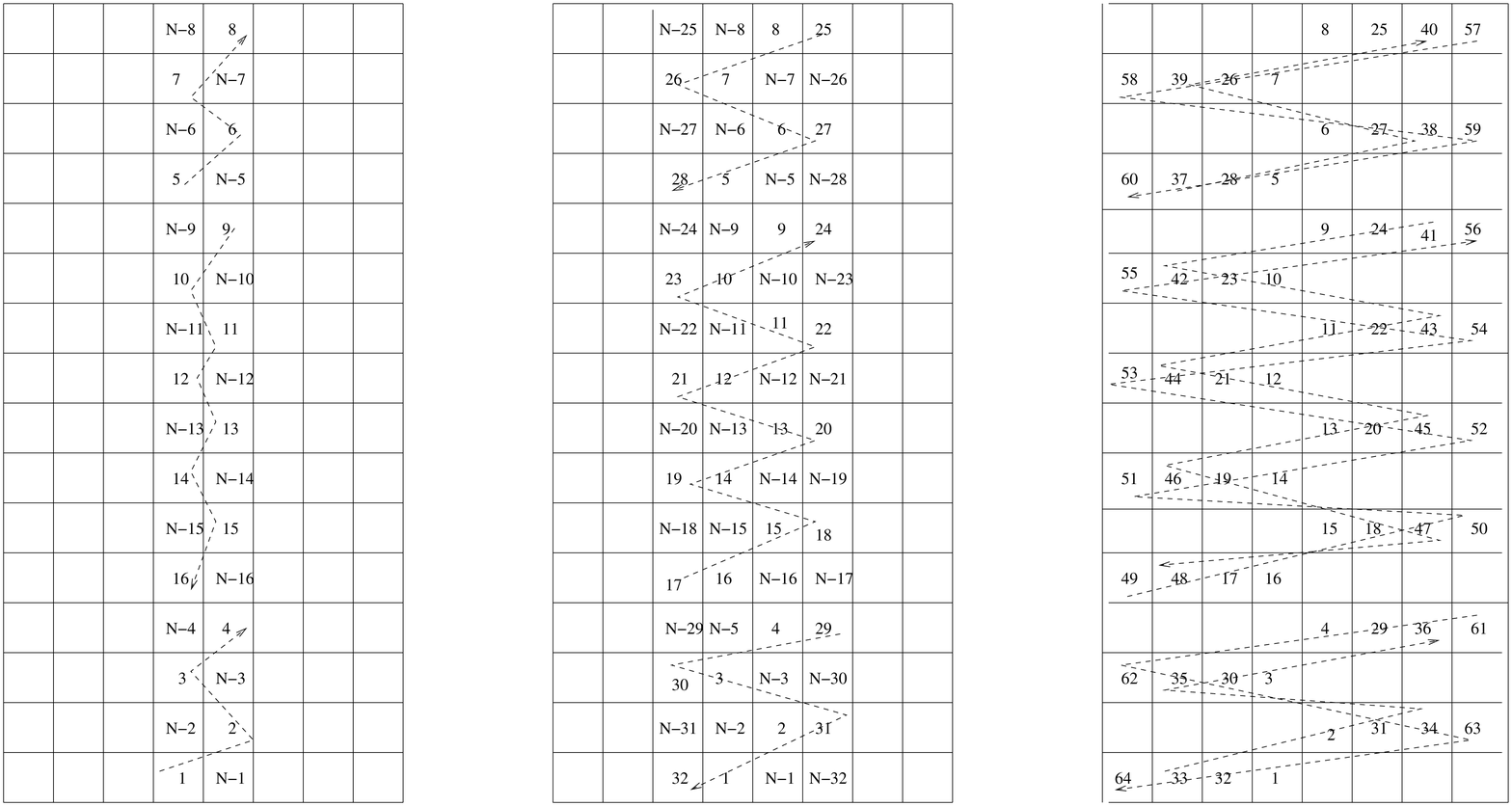} \label{leftfranklinfig}
\end{center}
 \end{figure}

\begin{table}
\caption{The left side of the Franklin square F3.} 
 \label{tableleftf3} 
\[
\begin{array}{|c|c|c|c|c|c|c|c|c|c|c|c|c|c|c|c|} \hline
N-57 & N-40 & N-25 & N-8 & 8 & 25     & 40 & 57 \\ \hline
58   & 39   & 26   & 7   & N-7 & N-26 & N-39 & N-58 \\ \hline
N-59 & N-38 & N-27 & N-6  & 6    & 27   & 38   & 59 \\ \hline
60   & 37   & 28   & 5    & N-5  & N-28 & N-37 & N-60 \\ \hline
N-56 & N-41 & N-24 & N-9  & 9    & 24   & 41   & 56 \\ \hline
55   & 42   & 23   & 10   & N-10 & N-23 & N-42 & N-55 \\ \hline
N-54 & N-43 & N-22 & N-11 & 11   & 22   & 43   & 54 \\ \hline
53   & 44   & 21   & 12  & N-12  & N-21 & N-44 & N-53 \\ \hline 
N-52 & N-45 & N-20 & N-13 & 13   & 20   & 45   & 52 \\ \hline
51   & 46   & 19   & 14  & N-14  & N-19 & N-46 & N-51 \\ \hline 
N-50 & N-47 & N-18 & N-15 & 15   & 18   & 47   & 50 \\ \hline
49   & 48   & 17   & 16  & N-16  & N-17 & N-48 & N-49 \\ \hline 
N-61 & N-36 & N-29 & N-4 & 4     & 29   & 36   & 61 \\ \hline
62   & 35   & 30   & 3   & N-3   & N-30 & N-35 & N-62 \\ \hline 
N-63 & N-34 & N-31 & N-2 & 2     & 31   & 34   & 63 \\ \hline
64   & 33   & 32   & 1   & N-1   & N-32 & N-33 & N-64 \\ \hline
\end{array} \]
\end{table}

\begin{figure}[h]
 \begin{center}
\caption{Constructing the right side of the Franklin square F3.}
     \includegraphics[scale=0.3]{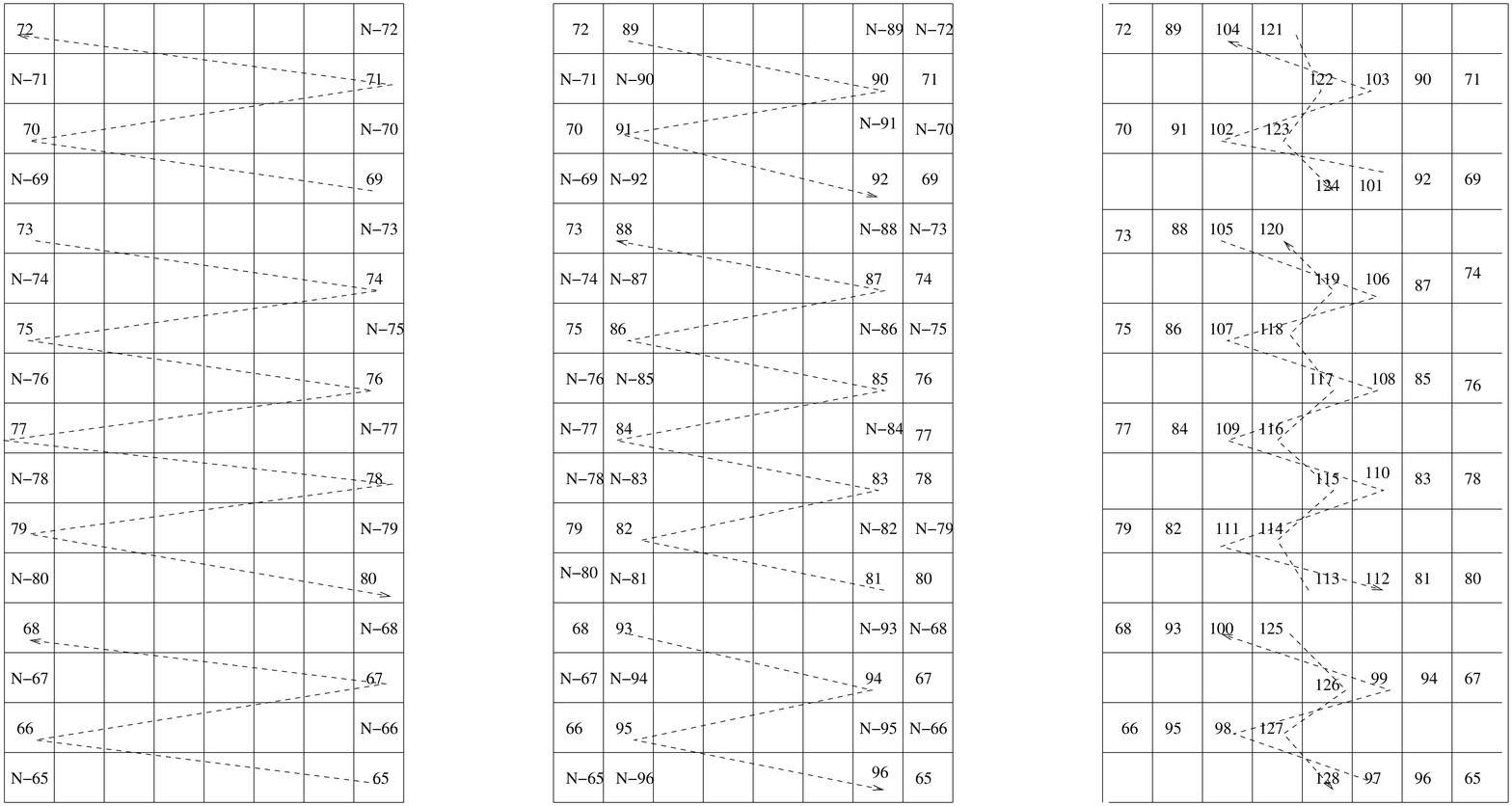} \label{rightfranklinfig}
\end{center}
 \end{figure}

\begin{table}
\caption{The right side of the Franklin square F3.} 
\label{rightfranklintable}
\[
\begin{array}{|c|c|c|c|c|c|c|c|c|c|c|c|c|c|c|c|} \hline
72   & 89   & 104   & 121   &N-121  & N-104 & N-89 & N-72 \\ \hline
N-71 & N-90 & N-103 & N-122 &122    &103    &90    &71 \\ \hline
70   & 91   & 102   & 123   &N-123  & N-102 & N-91 & N-70 \\ \hline
N-69 & N-92 & N-101 & N-124 &124    &101    &92    &69 \\ \hline
73   & 88   & 105   & 120   &N-120  & N-105 & N-88 & N-73 \\ \hline
N-74 & N-87 & N-106 & N-119 &119    &106    &87    &74 \\ \hline
75   & 86   & 107   & 118   &N-118  & N-107 & N-86 & N-75 \\ \hline
N-76 & N-85 & N-108 & N-117 &117    &108    &85    &76 \\ \hline
77   & 84   & 109   & 116   &N-116  & N-109 & N-84 & N-77 \\ \hline
N-78 & N-83 & N-110 & N-115 &115    &110    &83    &78 \\ \hline
79   & 82   & 111   & 114   &N-114  & N-111 & N-82 & N-79 \\ \hline
N-80 & N-81 & N-112 & N-113 &113    &112    &81    &80 \\ \hline
68   & 93   & 100   & 125   &N-125  & N-100 & N-93 & N-68 \\ \hline
N-67 & N-94 & N-99  & N-126 &126    &99     &94    &67 \\ \hline
66   & 95   & 98   & 127   &N-127  & N-98 & N-95 & N-66 \\ \hline
N-65 & N-96 & N-97  & N-128 &128    &97     &96    &65 \\ \hline

\end{array}
\]
\end{table}

\begin{prop} \label{verticalsumprop}
Let $A$ be  an $n \times n$ square  constructed by the above
procedure. Then $A$ is a Franklin square.
\end{prop}

\begin{proof}

Consider a column $c$ of $A$. Then, for the rows $n/4, 3n/4$ and $n$, if $c$ is odd, we have 
\[\begin{array}{lllllllll}
A_{n/4,c}+A_{n/4+1,c} = N+n/4, A_{3n/4,c}+A_{3n/4+1,c} = N-3n/4, 
A_{n,c}+A_{1,c} = N+n-1,
\end{array}
\]
and if $c$ is odd, we have

\[\begin{array}{lllllllll}
A_{n/4,c}+A_{n/4+1,c} = N-n/4, A_{3n/4,c}+A_{3n/4+1,c} = N+3n/4, A_{n,c}+A_{1,c} = N-(n-1).
\end{array} \]

Consider the remaining rows, that is,  $r \in \{1,2, \dots, n\} \setminus
\{n/4,  3n/4, n\}$.  In  the  top and  middle  parts of  the
square, that is, when $r < n/4$ or $r >n/2$, we have
\[\begin{array}{lllllllll}
A_{r,c}+A_{r+1,c} = N+1, & \mbox{if $c$ is odd, and  $r$ is odd,} \\
A_{r,c}+A_{r+1,c} = N-1, & \mbox{if $c$ is odd, and  $r$ is even,} \\
A_{r,c}+A_{r+1,c} = N-1, & \mbox{if $c$ is even, and  $r$ is odd,} \\
A_{r,c}+A_{r+1,c} = N+1, & \mbox{if $c$ is even, and  $r$ is even.} \\
\end{array}
\]

For the  middle part of the  square, that is, when  $n/4 < r
\leq n/2$, we have
\[\begin{array}{lllllllll}
A_{r,c}+A_{r+1,c} = N-1, & \mbox{if $c$ is odd, and  $r$ is odd,} \\
A_{r,c}+A_{r+1,c} = N+1, & \mbox{if $c$ is odd, and  $r$ is even,} \\
A_{r,c}+A_{r+1,c} = N+1, & \mbox{if $c$ is even, and  $r$ is odd,} \\
A_{r,c}+A_{r+1,c} = N-1, & \mbox{if $c$ is even, and  $r$ is even.} \\
\end{array}
\]

Thus, adjacent pair  of entries in the column  $c$ add  to $N \pm 1$
within a part of the square. At the boundaries, that is, for
the  rows  $n/4,  3n/4$  and  $n$,  the  sums  are  slightly
different.  Nevertheless, because  of the  alternating signs
between odd  and even  columns, it follows  that all  the $2
\times 2$ sub-squares add to $2N$, continuously. 

Also, for the  column $c$, the entries of  the top four rows
and the bottom four rows add to $(n/8)N+n/8$, if $c$ is odd,
and  $(n/8)N-n/8$, if $c$  is even.  Moreover, the  top four
rows  and the  last  four rows  of  the middle  part add  to
$(n/8)N-n/8$, if $c$ is odd,  and to $(n/8)N+n/8$, if $c$ is
even.  Consequently, all  the half  columns add  to $(n/4)N$
which is half the magic  sum. Therefore, all the columns add
to the magic sum. Since  the rows were constructed by adding
paired entries that add to $N$, it follows that all the half
rows and rows  add to the half the magic  and the magic sum,
respectively.

Pairs of entries of the column $c$ that are equidistant from
the center  add to the following  sums.

 If $c$ is odd, when  $i=0, \dots n/4-1$ (which restricts to
 the top and bottom parts of $A$), we have
\[A_{1+i,c}+A_{n-i,c} = N+(n/2-1-2i).
\]
For  the middle  part of  the  square, that is, when  $i=0,
\dots, n/4-1$ , we have
\[A_{n/2-i,c}+A_{n/2+1+i,c} = N-(1+2i).
\]

When $c$ is even: for the  top and bottom part of $A$, which
implies, $i=0, \dots n/4-1$, we have
\[A_{1+i,c}+A_{n-i,c} = N-(n/2-1-2i),
\]
and for the middle part, where  $i=0, \dots n/4-1$, we get 
\[A_{n/2-i,c}+A_{n/2+1+i,c} = N+(1+2i).
\]

Consequently, the entries of a right or left diagonal always add to the sum 
 \[
\sum_{i=0}^{n/4-1} N+(1+2i) +N-(1+2i) = \frac{n}{2}N.
\]

 Hence all the  left and right bend diagonals  always add to
 the magic sum.

Along, a  top or bottom  bend diagonal, two pairs  of adjacent
diagonal entries equidistant from the center, always, add to
$2N$ as follows.

For $r=1, \dots, n-1$ and $i=0,1, \dots, n/4-1$:
\[
(A_{r,1+2i} + A_{r+1,2+2i}) + (A_{r,n-2i}+A_{r+1,n-1-2i}) = 2N.
\]
\[
(A_{n,1} + A_{1,2}) + (A_{n,n}+A_{1,n-1}) = 2N. 
\]
Hence all the top diagonals add to the common magic sum, continuously.

For $r=n,n-1, \dots, 2$  and $i=0,1, \dots, n/4-1$:
\[(A_{r,1+2i}+A_{r-1,2+2i}) + (A_{r,n-2i}+A_{r-1,n-1-2i}) = 2N.
\]
\[
(A_{1,1}+A_{n,2}) + (A_{1,n}+A_{n,n-2}) = 2N.
\]
Therefore, all the bottom diagonals add to the common magic sum, continuously.

Thus, we conclude that $A$ is a Franklin square.

\end{proof}

\section{Maple Program to construct Franklin squares.} \label{codesection}

In this section, we provide a Maple procedure to construct a
$n \times n$ Franklin square.  The input to the procedure is
$n$. For  example, the command {\em  Franklin(16)} creates a
$16 \times 16$ Franklin square.

\begin{verbatim}
# nn is the order of the Franklin square
Franklin := proc(nn)
local A, n, N, i, j,t,s,r,cs,ca,e;
n:=eval(nn);
A := matrix(n,n,0);
N:=n*n+1;
#Start j loop
for j from 0 to n/8-1 do
# Constructing the left half of the square
#Bottom quarter going up
s:=1; r:=n; cs:=n/4; ca:=n/4+1; e:=n/8-1;
Up(A,n,N,s,r,cs, ca,e,j);
#Top quarter going up
s:=n/4+1; r:=n/4; cs:=n/4; ca:=n/4+1; e:=n/8-1; 
Up(A,n,N,s,r,cs,ca,e,j);
#Middle half going up 
s:=n+1; r:=3*n/4; cs:=n/4-1; ca:=n/4+2; e:=n/4-1; 
Up(A,n,N,s,r,cs,ca,e,j);
#Middle half going down
s:=n/2+1; r:=n/4+1; cs:=n/4; ca:=n/4+1; e:=n/4-1; 
Down(A,n,N,s,r,cs,ca,e,j);
#Top quarter going down
s:=3*n/2+1; r:=1; cs:=n/4-1; ca:=n/4+2; e:=n/8-1; 
Down(A,n,N,s,r,cs,ca,e,j);
#bottom quarter going down
s:=7*n/4+1; r:=3*n/4+1; cs:=n/4-1; ca:=n/4+2; e:=n/8-1; 
Down(A,n,N,s,r,cs,ca,e,j);


#Constructing the right half of the square 
#Bottom quarter going up
s:=(n/2)*(n/2)+1; r:=n; ca:=n/2+1; cs:=n; e:=n/8-1; 
Up(A,n,N,s,r,cs,ca,e,j);
#Top quarter going up
s:=(n/2)*(n/2)+n/4+1; r:=n/4; ca:=n/2+1; cs:=n; e:=n/8-1; 
Up(A,n,N,s,r,cs,ca,e,j);
#Middle half going up
s:=(n/2)*(n/2)+n+1; r:=3*n/4; ca:=n/2+2; cs:=n-1; e:=n/4-1; 
Up(A,n,N,s,r,cs,ca,e,j);
#Middle half going down
s:=(n/2)*(n/2)+n/2+1; r:=n/4+1; ca:=n/2+1; cs:=n; e:=n/4-1; 
Down(A,n,N,s,r,cs,ca,e,j);
#Top quarter going down
s:=(n/2)*(n/2)+3*n/2+1; r:=1; ca:=n/2+2; cs:=n-1; e:=n/8-1; 
Down(A,n,N,s,r,cs,ca,e,j);
#bottom quarter going down
s:=(n/2)*(n/2)+7*n/4+1; r:=3*n/4+1; ca:=n/2+2; cs:=n-1; e:=n/8-1; 
Down(A,n,N,s,r,cs,ca,e,j);
od; # end of j loop
print(A);
end;

# The Procedures Up and Down that are called from the Franklin main procedure

Up := proc(A,n,N,s,r,cs,ca,e,j)
local i,t;
for i from 0 to e do
t:=s+2*j*n+2*i; 
A[r-2*i,cs-2*j] := t;
A[r-2*i,ca+2*j]:= N-t;
A[r-1-2*i,cs-2*j] := N-(t+1);
A[r-1-2*i,ca+2*j]:= t+1; 
od;
end;

Down := proc(A,n,N,s,r,cs,ca,e,j)
local i,t;
for i from 0 to e do
t:=s+2*j*n+2*i; 
A[r+2*i,cs-2*j] := N-t;
A[r+2*i,ca+2*j]:= t;
A[r+1+2*i,cs-2*j] := t+1;
A[r+1+2*i,ca+2*j]:= N-(t+1); 
od;
end;







\end{verbatim}

\end{document}